\documentclass[12pt]{amsart}
\usepackage{amsfonts,amssymb,a4wide,url,mathscinet,mathdots,comment}

\newcommand{\Z}{\mathbb{Z}}
\newcommand{\St}{\operatorname{St}}
\newcommand{\ind}{\operatorname{ind}}
\newcommand{\vol}{\operatorname{vol}}

\newcommand{\GL}{\operatorname{GL}}
\newcommand{\diag}{\operatorname{diag}}
\newcommand{\OO}{\mathcal{O}}
\newcommand{\whit}{\mathcal{W}}
\newcommand{\val}{\operatorname{val}}

\newcommand{\ball}{\mathcal{B}}
\newcommand{\one}{1}
\newcommand{\rest}{|}
\newcommand{\bs}{\backslash}
\newcommand{\abs}[1]{\left|{#1}\right|}
\newcommand{\norm}[1]{\lVert#1\rVert}
\newcommand{\K}{\mathbf{K}}

\newtheorem{theorem}{Theorem}[section]
\newtheorem{lemma}[theorem]{Lemma}
\newtheorem{proposition}[theorem]{Proposition}%[subsection]
\newtheorem{corollary}[theorem]{Corollary}%[subsection]
\theoremstyle{remark}
\newtheorem{remark}[theorem]{Remark}%[subsection]

\begin{document}

\title[Support of matrix coefficients]{On the support of matrix coefficients of supercuspidal representations of the general linear group
over a local non-archimedean field}
\author{Erez Lapid}
\date{\today}

\begin{abstract}
We derive an upper bound on the support of matrix coefficients of suprecuspidal representations of the general linear group
over a non-archimedean local field. The results are in par with those which can be obtained from the Bushnell--Kutzko
classification of supercuspidal representations, but they are proved independently.
\end{abstract}

\maketitle

\section{Introduction}
%\section{Supercuspidal support property for $\GL_r$} \label{sec: supercuspidal support}
Throughout, let $F$ be a local non-archimedean field, $\OO$ its ring of integers, and let $G=\GL_r(F)$ be the general linear group of rank $r$
with center $Z$.
Let $\K=\GL_r(\OO)$ be the standard maximal compact subgroup of $G$ and for any $n\ge1$
let $K(n)=K_r(n)$ be the principal congruence subgroup, i.e.,
the kernel of the canonical map $\K\rightarrow\GL_r(\OO/\varpi^n\OO)$ where $\varpi$ is a uniformizer of $F$.
Denote by $\ball(n)$ the ball
\[
\ball(n)=\{g\in G:\norm{g},\norm{g^{-1}}\le q^n\}
\]
where $q$ is the size of the residue field of $F$ and $\norm{g}=\max\abs{g_{i,j}}_F$
where $g_{i,j}$ are the entries of $g$. Thus, $\{\ball(n):n\ge1\}$ is an open cover of $G$ by compact sets.
%and $\ball(n)\ball(m)\subset\ball(n+m)$.

The purpose of this short paper is to give a new proof of the following result.\footnote{Throughout,
by a representation of $G$ we always mean a complex, smooth representation.}

\begin{theorem} \label{thm: main}
Let $(\pi,V)$ be a supercuspidal representation of $G$ and let $(\pi^\vee,V^\vee)$ be its contragredient.
Let $v\in V$, $v^\vee\in V^\vee$ and assume that $v$ and $v^\vee$ are fixed under $K(n)$ for some $n\ge1$.
Then the support of the matrix coefficient $g\mapsto (\pi(g)v,v^\vee)$ is contained in $Z\ball(c(r)n)$
where $c(r)$ is an explicit constant depending on $r$ only.
\end{theorem}

As explained in \cite{MR3001800}, the theorem is a direct consequence of the classification of irreducible
supercuspidal representations
by Bushnell--Kutzko \cite{MR1204652}, or more precisely, of the fact that every such representation is induced from
a representation of an open subgroup of $G$ which is compact modulo $Z$. This is in fact known for many other cases
of reductive groups over local non-archimedean fields and in these cases it implies the analogue
of Theorem \ref{thm: main}. We refer the reader to \cite{MR3001800} for more details.

In contrast, the proof given here is independent of the classification.
It is based on two ingredients.
The first, which is special to the general linear group,
is basic properties of local Rankin--Selberg integrals for $G\times G$ which were defined and studied by
Jacquet--Piatetski-Shapiro--Shalika. In particular, we use an argument of Bushnell--Henniart, originally used to give an upper
bound on the conductor of Rankin--Selberg local factors \cite{MR1462836}.
The second ingredient is Howe's result on the integrality of the formal degree with respect to a suitable Haar measure
\cite{MR0342645}, a result which was subsequently extended to any reductive group \cite{MR1702257, MR1159104, MR1471867}.
The two ingredients are linked by the fact, which also follows from properties of Rankin--Selberg integrals,
that the formal degree is essentially the conductor of $\pi\times\pi^\vee$,
a feature that admits a conjectural generalization for any reductive group \cite{MR2350057}.
(For another relation between formal degrees and support of matrix coefficients see \cite{MR1198303}.)

%The analogue of Theorem \ref{thm: main} is expected to hold for any reductive group, and in fact it holds
%whenever

As explained in \cite{1504.04795, 1705.08191}, Theorem \ref{thm: main} is of interest for the problem of limit multiplicity.
%\Erez{Other applications?}

I would like to thank Stephen DeBacker, Tobias Finis, Atsushi Ichino and Julee Kim for useful discussions
and suggestions. I am especially grateful to Guy Henniart for his input leading to Remark \ref{rem: Henn}.

\section{A variant for Whittaker functions}
It is advantageous to formulate a variant of Theorem \ref{thm: main} for the Whittaker model.
Throughout, fix a character $\psi$ of $F$ which is trivial on $\OO_F$ but non-trivial on $\varpi^{-1}\OO_F$.
Let $N$ be the subgroup of upper unitriangular matrices in $G$.
If $\pi$ is a generic irreducible representation of $G$, we write $\whit^\psi(\pi)$ for its Whittaker
model with respect to the character $\psi_N$ of $N$ given by $u\mapsto\psi(u_{1,2}+\dots+u_{r-1,r})$.
Recall that every irreducible supercuspidal representation of $G$ is generic \cite{MR0404534}.

Let $A$ be the diagonal torus of $G$. For all $n\ge1$ let $A(n)$ be the open subset
\[
A(n)=\{\diag(t_1,\dots,t_r)\in A:q^{-n}\le\abs{t_i/t_{i+1}}\le q^n,\ \ i=1,\dots,r-1\}.
\]
Clearly, $ZA(n)=A(n)$ and $A(n)$ is compact modulo $Z$.

\begin{theorem} \label{thm: main'}
There exists a constant $c=c(r)$ with the following property.
Let $\pi$ be an irreducible supercuspidal representation of $G$ with Whittaker model $\whit^\psi(\pi)$ and $n\ge1$.
Then the support of any $W\in\whit^\psi(\pi)^{K(n)}$ is contained in $NA(cn)\K$.
\end{theorem}

In order to prove Theorem \ref{thm: main'} we set some more notation.
For any function $W$ on $G$ let
\begin{align*}
M_W&=\sup\{\val(\det g):W(g)\ne0\text{ and }\norm{g_r}=1\},\\
%\\&=\sup\{\val(\det t):t=\diag(t_1,\dots,t_r)\in A,\ t_r=1\text{ and }W(tk)\ne0\text{ for some }k\in\K\},\\
m_W&=\inf\{\val(\det g):W(g)\ne0\text{ and }\norm{g_r}=1\},
%\\&=\inf\{\val(\det t):t=\diag(t_1,\dots,t_r)\in A,\ t_r=1\text{ and }W(tk)\ne0\text{ for some }k\in\K\},
\end{align*}
(including possibly $\pm\infty$) where $g_r$ is the last row of $g$ and $\norm{(x_1,\dots,x_r)}=\max\abs{x_i}$.

Recall that by a standard argument (cf. \cite[Proposition 6.1]{MR581582})
for any right $K(n)$-invariant and left $(N,\psi_N)$-equivariant function $W$ on $G$,
if $W(tk)\ne0$ for some $t=\diag(t_1,\dots,t_n)\in A$ and $k\in\K$ then $\abs{t_i/t_{i+1}}\le q^n$, $i=1,\dots,r-1$.
Hence, $m_W\ge -{r\choose2}n$.

For any $W\in\whit^\psi(\pi)$ let $\widetilde W\in\whit^{\psi^{-1}}(\pi^\vee)$ be given by
$\widetilde W(g)=W(w_r\,^tg^{-1})$
where $w_r=\left(\begin{smallmatrix}&&1\\&\iddots&\\1\end{smallmatrix}\right)$.

Let $t=t(\pi)$ be the order of the group of unramified characters $\chi$ of $F^*$ such that $\pi\otimes\chi\simeq\pi$.
Clearly, $t$ divides $r$.

Let $f=f(\pi\times\pi^\vee)\in\Z$ be the conductor of the pair $\pi\times\pi^\vee$ (see below).

Theorem \ref{thm: main'} is an immediate consequence of the following two results.

\begin{proposition} \label{prop: mwMW}
For any $0\ne W\in\whit^\psi(\pi)$ we have
\[
M_W+m_{\widetilde W}=r-t-f.
\]
In particular, if $W\in\whit^\psi(\pi)^{K(n)}$ then
\[
M_W\le{r\choose2}n+r-t-f.
\]
\end{proposition}

\begin{proposition} \label{prop: lwrbndf}
We have
\[
f\ge (r+1)r-2(t+v_q(t))
\]
where $v_q(t)$ is the maximal power of $q$ dividing $t$.
Moreover, $f$ is even if $q$ is not a square.
\end{proposition}

\begin{proof}[Proof of Proposition \ref{prop: mwMW}]
The argument is inspired by \cite{MR1462836}.

We may assume without loss of generality that $\pi$ is unitarizable.
For any $\Phi\in\mathcal{S}(F^r)$ consider the local Rankin--Selberg integral
\[
A^{\psi}(s,W,\Phi)=\int_{N\bs G}\abs{W(g)}^2\Phi(g_r)\abs{\det g}^s\ dg,
\]
a Laurent series in $x=q^{-s}$ which represents a rational function in $x$ \cite{MR701565}.
Note that if $\Phi(0)=0$ then $A^{\psi}(s,W,\Phi)$ is a Laurent polynomial in $x$ since $W$ is compactly
supported modulo $ZN$.
Also note that for any $\lambda\in F^*$ we have
\[
A^{\psi}(s,W,\Phi(\lambda\cdot))=\abs{\lambda}^{-rs}A^\psi(s,W,\Phi).
\]

Recall the functional equation (\cite[Theorem 2.7 and Proposition 8.1]{MR701565} together with \cite{MR1703873})
\[
q^{f(\frac12-s)}(1-q^{-ts})A^{\psi}(s,W,\Phi)=(1-q^{t(s-1)})A^{\psi^{-1}}(1-s,\widetilde W,\tilde\Phi)
\]
where
\[
\tilde\Phi(x)=\int_{F^n}\Phi(y)\psi(x\,^ty)\ dy
\]
is the Fourier transform of $\Phi$ and $f\in\Z$ is the conductor.

%Write
%\[
%A^{\psi}(s,W,\Phi)=\sum_{m\in\Z}a_m(W,\Phi)x^m
%\]
%where
%\begin{equation} \label{def: am}
%a_m(W,\Phi)=\int_{g\in N\bs G:\val(\det g)=m}\abs{W(g)}^2\Phi(g_r)\ dg.
%\end{equation}

Now let $\Phi_0$ be the characteristic function of the standard lattice $\{\xi\in F^r:\norm{\xi}\le1\}$ and set
$A^{\psi}_0(s,W)=A^\psi(s,W,\Phi_0)$. Then $\tilde\Phi_0=\Phi_0$ and we obtain
\[
q^{f(\frac12-s)}(1-q^{-ts)})A^{\psi}_0(s,W)=
(1-q^{t(s-1)})A^{\psi^{-1}}_0(1-s,\widetilde W).
\]
Let $\Phi_1=\Phi_0-\Phi_0(\varpi^{-1}\cdot)$ be the characteristic function of the $\K$-invariant set
$\{\xi\in F^r:\norm{\xi}=1\}$ of primitive vectors. Then
$A^{\psi}_1(s,W):=A^\psi(s,W,\Phi_1)=(1-q^{-rs})A^{\psi}_0(s,W)$ and thus,
\[
q^{f(\frac12-s)}\frac{1-q^{-ts}}{1-q^{-rs}}A^{\psi}_1(s,W)=
\frac{1-q^{t(s-1)}}{1-q^{r(s-1)}}A^{\psi^{-1}}_1(1-s,\widetilde W).
\]
Recall that $A^\psi_1(s,W)$ is a Laurent polynomial $P^\psi_w(x)$ in $x=q^{-s}$. We get an equality of Laurent
polynomials
\begin{equation} \label{eq: polyeq}
q^{\frac12f}x^{f}\frac{1-x^{t}}{1-x^r}P^\psi_W(x)=
\frac{1-y^{t}}{1-y^r}P^{\psi^{-1}}_{\widetilde W}(y)
\end{equation}
where $y=q^{-1}x^{-1}$. Note that the fact that $P^\psi_W(x)$ is divisible by $\frac{1-x^r}{1-x^{t}}$
amounts to saying that the integral
\[
\int_{g\in N\bs G:\norm{g_r}=1,\val(\det g)\equiv a\pmod{r/t}}\abs{W(g)}^2\ dg
\]
is independent of $a$, which in turn follows from (and in fact, equivalent to) the fact that $W$ is orthogonal to
$W\abs{\det}^{\frac{2\pi\textrm{i}j}{r\log q}}$ unless $j$ is divisible by $r/t$.
Also note that $P^\psi_W$ has non-negative coefficients and the degree of $P^\psi_W$ is $M_W$.
Likewise, the degree of $P^\psi_W(y)$ as a Laurent polynomial in $x$ (i.e., the order of pole of $P^\psi_W$ at $0$)
is $-m_W$.
Comparting degrees in \eqref{eq: polyeq} we obtain Proposition \ref{prop: mwMW}.
\end{proof}

\begin{proof}[Proof of Proposition \ref{prop: lwrbndf}]
By an argument based on properties of Rankin-Selberg integrals,
the formal degree, with respect to a suitable choice of Haar measure, is related to $f$ by the formula
\[
d_\pi=\frac{t}{r}q^{\frac12 f}\frac{1-q^{-1}}{1-q^{-t}}
\]
(\cite[Theorem 2.1]{MR3649356}).
Comparing it to the formal degree of the Steinberg representation $\St$ with respect to the same measure we get
\[
\frac{d_\pi}{d_{\St}}=t\cdot q^{t-{r+1\choose 2}+\frac12 f}\cdot\frac{q^r-1}{q^t-1}.
\]
On the other hand, $\frac{d_\pi}{d_{\St}}$ (which is independent of the choice of Haar measure)
is a (positive) integer (cf.~\cite{MR0342645}, \cite{MR607380}, \cite[Appendice 3]{MR743063}).
%Although measure in [loc. cit.] is not made explicit, it is easy to see
%that the two measures differ by a constant of the form $q^\alpha t$ where $t$ is coprime to $q$
%and $\alpha$ is bounded above and below in terms of $r$.
The lemma follows.
\end{proof}

\begin{remark} \label{rem: Henn}
As explained to me by Guy Henniart, the lower bound in Proposition \ref{prop: lwrbndf} is not sharp.
In fact, a precise formula for $f(\pi\times\pi^\vee)$, and more generally for
$f(\pi_1\times\pi_2)$ for an arbitrary pair of irreducible supercuspidal representations $\pi_i$
of $\GL_{n_i}(F)$, $i=1,2$ is given in \cite{MR1606410}.
The expression is in terms of the Bushnell--Kutzko description of supercuspidal representations.
Using this, one can sharpen Proposition \ref{prop: lwrbndf} as follows.

First note that $f(\pi\times\pi^\vee)$ is insensitive to twisting $\pi$ by a character.
Suppose that $\pi$ is minimal under twists, i.e., $f(\pi\times\chi)\ge f(\pi)$ for any character
$\chi$ of $F^*$ where $f(\pi)$ is the conductor of $\pi$.
Then by \cite[Lemma 3.5]{MR3711826} and its proof we have $f(\pi\times\pi^\vee)\ge r^2-t+\frac12r(f(\pi)-r)$
with equality if $f(\pi)=r$, i.e., if $\pi$ has level $0$ (that is, if $\pi$ has a non-zero vector invariant under $K(1)$,
in which case $t=r$). Thus, if $f(\pi)>r+1$ then we get $f(\pi\times\pi^\vee)\ge r(r+1)-t$.
On the other hand, we always have $f(\pi)\ge r$ \cite[(5.1)]{MR882297} and
if $f(\pi)=r+1$, i.e., if $\pi$ is epipelagic then $t=1$ and $f(\pi\times\pi^\vee)=(r-1)(r+2)$.
More generally, if $f(\pi)$ is coprime to $r$, i.e., if $\pi$ is a Carayol representation, then $t=1$ and
$f(\pi\times\pi^\vee)=(r-1)(f(\pi)+1)$. For instance, this follows from \cite[(6.1.1),(6.1.2)]{MR1606410} and [ibid., Theorem 6.5(i)]
where in its notation we have $n=e=d=r$ and $\mathfrak{c}(\beta_1)=m(r-1)$ --
cf. second paragraph (``minimal case'') of [ibid., p.~727] with $k=m$, $e(\gamma)=d=r$.
%\[
%q^{\mathfrak{c(\beta)}}=\frac{[(\mathfrak{H}^1)^*:\mathfrak{J}^1]}{[\mathfrak{o}_E:\mathfrak{p}_E]}=
%[(\mathfrak{H}^1)^*:\mathfrak{J}^1]q^{-d/e}
%\]
%(see proof of Lemma 6.4 in [ibid.]) while in the case at hand we have $d=e=r$, $\mathfrak{J}^1=\mathfrak{H}^1=\mathfrak{P}(E)$
%and thus $(\mathfrak{H}^1)^*=\mathfrak{A}(E)$ (cf.~\cite[\S2.1]{MR3158004}). Hence, the claim follows from.

To conclude, for any supercuspidal $\pi$ we have
\begin{equation} \label{eq: sharp}
f(\pi\times\pi^\vee)\ge r(r+1)-2t
\end{equation}
with equality if and only if $\pi$ is a twist of either a representation of level $0$ or an epipelagic representation.

%Note that \cite[Lemma 3.5]{MR3711826} uses the Bushnell--Kutzko classification of supercuspidal representations.

%As communicated to me by Guy Henniart, one can infer from this formula the bound $f\ge (r+1)r-2t$,
%which is sharp for representations of level $0$ as well as for epipelagic representations.
%The formula also implies that $f$ is even (without restriction on $q$).

Alternatively, one could infer \eqref{eq: sharp} and the conditions for equality
from the local Langlands correspondence for $G$ (cf.~\cite{MR2730575}).
Details will be appear in the upcoming thesis of Kilic.

The results of \cite{MR1606410} also give that $f(\pi\times\pi^\vee)$ is even.
(Details will be given elsewhere.) In the Galois side this follows from a result of Serre \cite{MR0321908}.

I am grateful to Guy Henniart for providing me this explanation and allowing me to include it here.

For our purpose, the precise lower bound on $f(\pi\times\pi^\vee)$ is immaterial -- it is sufficient to have
the inequality $f\ge0$ (or even, $f\ge c'n$ for some fixed $c'$ depending only $r$).
The point is that we do not use either the local Langlands correspondence or the classification of supercuspidal
representations (but of course, we do use the non-trivial analysis of \cite{MR0342645}
which depends on \cite{MR0492088}).

Nonetheless, it would be interesting to prove \eqref{eq: sharp} (and perhaps the evenness of $f(\pi\times\pi^\vee)$)
without reference to the classification or to the local Langlands correspondence.
%However, this does not seem to be any easier to prove.

%The latter, however, lies far deeper than the results quoted above.

%\item
%\end{enumerate}
\end{remark}

\section{Proof of main result}
In order to deduce Theorem \ref{thm: main} from Theorem \ref{thm: main'} we make the argument of
\cite[Proposition 2.11]{MR3267120} effective in the case of $G=\GL_r$.

For any $t=\diag(t_1,\dots,t_r)\in A$ consider the compact open subgroup
\[
N(t)=N\cap t\K t^{-1}=\{u\in N:\val(u_{i,j})\ge\val(t_i)-\val(t_j)\text{ for all }i<j\}
\]
of $N$. Set
\[
t_0=\diag(\varpi,\varpi^2,\dots,\varpi^{2^{r-1}})\in A.
\]

\begin{proposition} \label{prop: effstab}
Let $f$ be a compactly supported continuous funciton on $G$.
Assume that $f$ is bi-invariant under $K_{r-1}(n)$ for some $n\ge1$.
Then
\[
\int_Nf(u)\psi_N(u)\ du=\int_{N(t_0^n)}f(u)\psi_N(u)\ du.
\]
In particular, $\int_Nf(u)\psi_N(u)\ du=0$ if $f$ vanishes on $\ball((2^{r-1}-1)n)$.
\end{proposition}

We first need some more notation.
Write $N=U\ltimes V$ where $U=N_{r-1}$ is the group of unitriangular matrices in $\GL_{r-1}$
embedded in $\GL_r$ in the upper left corner and $V$ is the unipotent radical
of the parabolic subgroup of type $(r-1,1)$, i.e., the $r-1$-dimensional abelian group
\[
V=\{u\in N:u_{i,j}=0\text{ for all }i<j<r\}.
\]
For any $\underline{\alpha}=(\alpha_1,\dots,\alpha_{r-1})\in F^{r-1}$ let $x(\underline{\alpha})$
be the $r\times r$-matrix that is the identity except for the $(r-1)$-th row which is $(\alpha_1,\dots,\alpha_{r-1}+1,0)$.
Of course $x$ is not a group homomorphism because of the diagonal entry.

For $t\in A$ write $U(t)=U\cap N(t)$ and $V(t)=V\cap N(t)$ so that $N(t)=U(t)\ltimes V(t)$.
The following is elementary.
\begin{lemma} \label{lem: aux123}
For any $n\ge1$ let $L(n)$ be the lattice of $F^{r-1}$ given by
\[
L(n)=\{(\alpha_1,\dots,\alpha_{r-1}):\val(\alpha_j)\ge n(2^{r-1}-2^{j-1}),\ j=1,\dots,r-1\}.
\]
Then
\begin{enumerate}
\item \label{part: canconj}
$x(\underline{\alpha}),x(\underline{\alpha})^u\in K_{r-1}(n)$ for any $\underline{\alpha}\in L(n)$ and $u\in U(t_0^n)$.
\item \label{part: charcan}
For any $v\in V$ we have
\[
\vol(L(n))^{-1}\int_{L(n)}\psi_N(v^{x(\underline{\alpha})})\ d\underline{\alpha}=
\begin{cases}\psi_N(v)&v\in V(t_0^n),\\0&\text{otherwise.}\end{cases}
\]
\end{enumerate}
\end{lemma}

\begin{proof}[Proof of Proposition \ref{prop: effstab}]
We prove the statement by induction on $r$. The case $r=1$ is trivial.
For the induction step, note that if $f$ is bi-invariant under $K_{r-1}(n)$ then
the function $h=\int_Vf(\cdot v)\psi_N(v)\ dv$ on $\GL_{r-1}$ is bi-$K_{r-2}(n)$-invariant
(since $\GL_{r-2}$ normalizes the character $\psi_N\rest_V$).
Therefore, by induction hypothesis we have
\begin{multline*}
\int_Nf(u)\psi_N(u)\ du=\int_Uh(u)\psi_N(u)\ du=
\int_{U(t_0^n)}h(u)\psi_N(u)\ du\\=\int_V\int_{U(t_0^n)}f(uv)\psi_N(uv)\ du\ dv.
\end{multline*}
Now we use Lemma \ref{lem: aux123}.
By part \ref{part: canconj}, Since $f$ is bi-$K_{r-1}(n)$-invariant, for any $\underline{\alpha}\in L(n)$
we can write the above as
\[
\int_V\int_{U(t_0^n)}f(ux(\underline{\alpha})vx(\underline{\alpha})^{-1})\psi_N(uv)\ du\ dv=
\int_V\int_{U(t_0^n)}f(uv)\psi_N(uv^{x(\underline{\alpha})})\ du\ dv.
\]
Averaging over $\underline{\alpha}\in L(n)$ and using part \ref{part: charcan},
we may replace the integration over $V$ by integration over $V(t_0^n)$. This yields the induction step.
\end{proof}

Let $\Pi_{\psi}=\ind_N^G\psi$. For $\varphi\in\Pi_{\psi}$ and $\varphi^\vee\in\Pi_{\psi^{-1}}$ let
\[
(\varphi,\varphi^\vee)_{N\bs G}=\int_{N\bs G}\varphi(g)\varphi^\vee(g)\ dg.
\]
Also set,
\[
A^\circ(n)=A\cap\ball(n)=\{\diag(t_1,\dots,t_r)\in A:q^{-n}\le\abs{t_i}\le q^n,\ \ i=1,\dots,r\}.
\]
Proposition \ref{prop: effstab}, together with the argument of \cite[Proposition 2.12]{MR3267120}, which
was communicated to us by Jacquet, yield the following.

\begin{corollary} \label{cor: suppindmc}
There exists a constant $c$, depending only on $r$ with the following property.
Assume that $\varphi\in\Pi_\psi^{K(n)}$ and $\varphi^\vee\in\Pi_{\psi^{-1}}^{K(n)}$
are both supported in $NA^\circ(n)\K$ for some $n\ge1$.
Then $(\Pi_\psi(\cdot)\varphi,\varphi^\vee)_{N\bs G}$ is supported in the ball $\ball(cn)$.
\end{corollary}

Indeed, the function $M(g)=(\Pi_\psi(\cdot)\varphi,\varphi^\vee)_{N\bs G}$ is clearly bi-$K(n)$-invariant.
Let $f$ be a compactly supported bi-$K(n)$-invariant function on $G$. By Fubini's theorem
\begin{multline*}
\int_GM(g)f(g)\ dg=\int_G\int_{N\bs G}\varphi(xg)\varphi^\vee(x)f(g)\ dg\ dx=
\int_{N\bs G}\int_G\varphi(xg)\varphi^\vee(x)f(g)\ dg\ dx\\=
\int_{N\bs G}\int_G\varphi(g)\varphi^\vee(x)f(x^{-1}g)\ dg\ dx=
\int_{N\bs G}\int_{N\bs G}\varphi^\vee(x)\varphi(y)K_f(x,y)\ dy\ dx
\end{multline*}
where
\[
K_f(x,y)=\int_Nf(x^{-1}ny)\psi_N(n)\ dn.
\]
From Proposition \ref{prop: effstab} we infer that there exists $c$, depending only on $r$, such that
if $f$ vanishes on $\ball(cn)$ then $K_f(x,y)=0$ for all $x,y\in\ball(n)$. The corollary follows.

%\Erez{Should we repeat the proof?}

Finally, we prove Theorem \ref{thm: main}.
\begin{proof}[Proof of Theorem \ref{thm: main}]
Let $\pi$ be a supercuspidal irreducible representation of $G$.
Suppose that $W\in\whit^{\psi}(\pi)^{K(n)}$ and $W^\vee\in\whit^{\psi^{-1}}(\pi^\vee)^{K(n)}$.
By Theorem \ref{thm: main'}, both $W$ and $W^\vee$ are supported in $NA(cn)\K$ for suitable $c$.
Upon modifying $c$, we may write $W(g)=\int_ZW_0(zg)\omega_{\pi}^{-1}(z)\ dz$
(with $\vol(Z\cap\K)=1$) where $W_0\in\Pi_{\psi}^{K(n)}$ is supported in $NA^\circ(cn)\K$.
For instance we can take $W_0=W\one_X$ where $X$ is the set
\[
X=\{g\in G:0\le\val(\det g)<r\}.
\]
Similarly, write $W^\vee(g)=\int_ZW_0^\vee(zg)\omega_{\pi}(z)\ dz$ where $W_0^\vee\in\Pi_{\psi^{-1}}^{K(n)}$
is supported in $NA^\circ(cn)\K$.
Then up to a scalar
\[
(\pi(g)W,W^\vee)=\int_{ZN\bs G}W(xg)W^\vee(x)\ dx=
\int_Z(\Pi_\psi(zg)W_0,W_0^\vee)_{N\bs G}\ dz.
\]
The result therefore follows from Corollary \ref{cor: suppindmc}.
\end{proof}

\def\cprime{$'$} 
\providecommand{\bysame}{\leavevmode\hbox to3em{\hrulefill}\thinspace}
\providecommand{\MR}{\relax\ifhmode\unskip\space\fi MR }
\providecommand{\MRhref}[2]{%
  \href{http://www.ams.org/mathscinet-getitem?mr=#1}{#2}
}
\providecommand{\href}[2]{#2}

\bibliography{../Bibfiles/all}

\end{document}